\documentclass[12pt]{amsart}

\usepackage{amssymb,bm,mathrsfs,enumerate}

\usepackage{color}

\usepackage[centering, width=6in]{geometry}
\usepackage[colorlinks,pagebackref]{hyperref}

\newtheorem{thm}{Theorem}[section]
\newtheorem*{thm*}{Theorem}
\newtheorem{lem}[thm]{Lemma}
\newtheorem{prop}[thm]{Proposition}
\newtheorem{cor}[thm]{Corollary}
\newtheorem{conj}[thm]{Conjecture}

\theoremstyle{definition}
\newtheorem{defn}[thm]{Definition}
\newtheorem{prob}[thm]{Problem}

\theoremstyle{remark}
\newtheorem{rem}[thm]{Remark}

\numberwithin{equation}{section}

\DeclareMathOperator{\bexp}{exp}

\DeclareMathOperator{\mo1}{mod}

\newcommand{\N}{\mathbb N}

\newcommand{\Q}{\mathbb Q}

\newcommand{\Z}{\mathbb Z}

\begin{document}
	
	\title{On $\beta$-adic expansions of powers of an algebraic integer omitting a digit}
	
		\author{Jiuzhou Zhao}

    \address{School of Mathematical Sciences,  Key Laboratory of MEA(Ministry of Education) $\&$ Shanghai Key Laboratory of PMMP,  East China Normal University, Shanghai, 200241, China}

    \email{zhao9zone@gmail.com}
	
	\author{Ruofan Li$^*$}
	
	\address{Department of Mathematics, Jinan University, Guangzhou, 510632, China}
	
	\email{liruofan@jnu.edu.cn}
	
	
	
\thanks{$^*$Corresponding author.}

	\begin{abstract}
	Let $\alpha, \beta$ be two relatively prime algebraic integers in a number field $K$ and $N$ be a positive integer. We show that the number of $n\in\{1,2,\dots,N\}$ such that the $\beta$-adic expansion of $\alpha^n$ omits a given digit is less than $C_1 N^{\sigma(\beta)}$, where $\sigma(\beta):=\frac{\log(|N(\beta)|-1)}{\log|N(\beta)|}$ and $C_1$ is a constant depending only on $\beta$, if all prime ideal factors of $\beta$ are unramified and their norms are integer primes.
	\end{abstract}
	
	\maketitle
	
	\section{Introduction}\label{sc:intro}
	
	Consider the ternary expansion
	\[(2^n)_3:=a_{k_n}\dots a_1a_0,\]
	where $a_j\in\{0,1,2\}$, $0 \leq j \leq k_{n}$ satisfies $2^n=\sum_{j=0}^{k_n}a_j 3^j$. It is an interesting phenomenon that $(2^0)_3=1$, $(2^2)_3=11$ and $(2^8)_3=100111$ omit the digit $2$. No other value of $n$ such that $(2^n)_3$ omits the digit $2$ is known. Indeed, Erd\H{o}s \cite{Erdoes1979} proposed the following conjecture, which is still open. 
	
	\begin{conj}\label{2404152202}
		The ternary expansion of $2^n$ can not omit the digit $2$ for all $n\ge9$.
	\end{conj}
	
	This conjecture is related to the \emph{persistence} problem (see \cite{Bonuccelli2020,Faria2014}) which concerns base $b$ expansion of natural numbers. 
	Given an integer $b>1$ and a natural number $n=\sum_{j=1}^{k}d_j b^{k-j}$ with $d_j\in\{0,1,\dots,b-1\}$, define the Sloane map $S_b:\N\to\N$ by $S_b(n):=\prod_{j=1}^{k}d_j$. By \cite[Proposition 1.1]{Faria2014}, $S_b(n)<n$ for all $n \ge b$. Thus, the orbit under the Sloane map $S^m_b(n)$, $m\ge1$
	always stabilizes after a finite number of steps, that is, there exists a minimal number $l_b(n)$ such that $S^j_b(n)=S^{l_b(n)}_b(n)$ for all $j\ge l_b(n)$. When $b=2$, it is trivial to see that $l_{b}(n) = 1$ for all $n$, the \emph{persistence} problem asks whether a uniform bound of $l_{b}(n)$ exists in general.
	\begin{prob}[\emph{Persistence} problem]
		 For a given $b>2$, is there a positive number
		$B(b)$ such that $l_b(n)\le B(b)$ for all n?
	\end{prob}
	In the case of base $b=3$, the only nonzero values assumed by the Sloane map are powers of $2$. Hence, in order to answer the persistence problem
	for base $3$, it suffices to establish the following weaker form of Conjecture \ref{2404152202}. 
	\begin{conj}
		There is a positive integer $k_0$ such that for all $k\ge k_0$, the ternary expansion of $2^k$ can not omit the digit $0$. 
	\end{conj}
		
	Another problem related to Conjecture \ref{2404152202} is determining \emph{practical binomial coefficients} (see \cite{Leonetti2020,Sanna2021}). A positive integer $n$ is called \emph{practical} if all positive integers less than $n$ can be written as a sum of distinct divisors of $n$. Leonetti and Sanna \cite{Leonetti2020} remarked that, likely, there are only finitely many positive integers $n$ such that $\binom{2n}{n}$ is not a practical number. They proved that if $n$ is a power of $2$ whose ternary expansion omits the digit $2$, then $\binom{2n}{n}$ is not a practical number \cite[Proposition 2.1]{Leonetti2020}. 
	
	Progress towards Conjecture \ref{2404152202} has been in the form of upper bounds on 
	\[\mathcal M(N):=\#\big\{1\le n\le N\colon (2^n)_{3} \text{ omits the digit $2$} \big\},\]
	where the symbol $\#$ denote cardinality. The best known bound on $\mathcal M(N)$ is due to Narkiewicz \cite{Narkiewicz1980} who proved that
	\begin{equation}\label{2402121459}
		\mathcal M(N)\le1.62N^{\sigma}, \quad\text{	where $\sigma:=\log_32\approx0.63092$. }
	\end{equation}
	We refer the reader to \cite{Dupuy2016a,Holdum2015,Kennedy2000,Lagarias2009} for more results related to Narkiewicz's result. 
	
	In this paper, we are going to generalize Narkiewicz's result \eqref{2402121459} by describing the above phenomena in general algebraic number fields. Let $K$ be a number field with ring of integers $\mathcal O_K$. Fix an element $\beta\in \mathcal O_K$ with norm  $|N(\beta)|>1$.
	\begin{defn}
		We call $(\beta,\,\mathcal \{0,1,\dots,|N(\beta)|-1\})$ a \emph{canonical number system (CNS)} in $\mathcal O_K$, if every $\alpha\in\mathcal O_K$ can be represented uniquely as 
		\begin{equation}
			\alpha=a_0+a_1\beta+\cdots+a_m \beta^m,\quad a_j\in\mathcal \{0,1,\dots,|N(\beta)|-1\} \;(j=0,1,\dots,m),
		\end{equation}
		which is called the \emph{radix expansion} of $\alpha$ in base $\beta$. For convenience, denote 
		\begin{equation}
			(\alpha)_{\beta}:=a_m\dots a_{1}a_0,\text{ and } (\alpha)_{\beta,j}:=a_j\;(j=0,1,\dots,m). 
		\end{equation}
	\end{defn}
	
    For $b \in \{0,1,\dots,|N(\beta)|-1\}$, denote
	\begin{equation}\label{2402141036}
	\mathcal M_b(\alpha,\beta,N):=\#\big\{1\le n\le N\colon (\alpha^n)_{\beta,j}\neq b\;\text{for all possible $j$}\big\},
	\end{equation}
	For the rest of this article, we assume $\alpha$ is not a root of unity as otherwise $\alpha^n$ only have finitely many different value as $n$ changes. Recall that $\alpha,\,\beta\in\mathcal O_K$ are \emph{relatively prime} if the prime ideal decomposition 
	\begin{equation}\label{2404121129}
		(\beta)=\mathfrak p_1^{e_1}\cdots\mathfrak p_{h}^{e_h}
	\end{equation}
	satisfies that $\mathfrak p_j\nmid\alpha$ (i.e. $\alpha\notin\mathfrak p_j $) for all $j=1,2,\dots,h$. Our first result is an upper bound of $\mathcal M_b(\alpha,\beta,N)$ (similar to \eqref{2402121459}).
		
	\begin{thm} \label{thm:2402262104}
	 Suppose $(\beta,\,\{0,1,\dots,|N(\beta)|-1\})$ is a CNS, $\beta$ is not divided by ramified primes and $\alpha$ is relatively prime to $\beta$, then 
	 \begin{equation}
	 \mathcal M_b(\alpha,\beta,N)\le  C_1 N^{\sigma(\beta)}
	 \end{equation} 
	 holds for any digit $b\in\{1,\dots,|N(\beta)|-1\}$, where $\sigma(\beta):=\frac{\log(|N(\beta)|-1)}{\log|N(\beta)|}$ and $C_1$ is a constant depending only on $\beta$.
	\end{thm}
		
	Taking $\mathcal O_K=\Z$, $\alpha=2$ and $\beta=3$, Theorem \ref{thm:2402262104} leads to \eqref{2402121459} up to a constant multiple. 
	
	K\'atai and Szab\'o \cite{Katai1975} determined all the CNS for Gaussian integers. And the question of determining all CNS in quadratic number fields has been answered by \cite{Katai1980,Katai1981}. However, in extensions of higher degree, there is not necessarily a CNS. We say $\mathcal O_K$ is \emph{monogenic} if there exists $\gamma\in\mathcal O_K$, such that $\{1,\gamma,\dots,\gamma^{d-1}\}$ is an integer basis in $\mathcal O_K$. It is clear from the definition that if $(\beta,\,\{0,1,\dots,|N(\beta)|-1\})$ is a CNS in $\mathcal O_K$, then $\mathcal O_K = \Z[\beta]$, hence $\mathcal O_K$ must be monogenic. Although $\mathcal O_K = \Z[\beta]$ does not implies $(\beta,\,\{0,1,\dots,|N(\beta)|-1\})$ is a CNS in $\mathcal O_K$ in general, we do have the following criterion to determine whether the ring of integers has a CNS.
	\begin{thm}[Kov\'acs \cite{Kovacs1981}]
		Let $K$ be a finite extension of $\Q$ with ring of integers $\mathcal O_K$ and $[K:\Q]=d\ge3$. There exists a CNS $(\beta,\,\{0,1,\dots,|N(\beta)|-1\})$ in $\mathcal O_K$ if and only if $\mathcal O_K$ is monogenic. 
	\end{thm}
	
	However, for number fields with degree at least $3$, their rings of integers are unlikely to be monogenic, see \cite{JK17,Li22,Smith21} for some recent results on monogeneity of number fields. In order to study those non-monogenic number fields, we introduce the concept of \emph{$\beta$-adic expansion}	which is a natural generalization of $p$-adic expansion.
	\begin{defn}
	Given a number field $K$ and its ring of integers $\mathcal{O}_{K}$. Fix $\beta\in \mathcal O_K$ with norm $|N(\beta)|>1$ and a set of representatives $\mathcal D_{\beta}$ of the quotient group $\mathcal{O}_{K} / \beta \mathcal{O}_{K}$. For every $\alpha \in \mathcal{O}_{K}$, the \emph{$\beta$-adic expansion} of $\alpha$ (with respect to $\mathcal D_{\beta}$) is the unique sequence $(a_i)_{i\in\N} \in \mathcal D_{\beta}^{\N}$ such that 
	\begin{equation}\label{2401171345}
		\alpha=\lim_{i\rightarrow\infty}a_0+\cdots+a_i\beta^i
	\end{equation}
    with respect to $\mathfrak{p}$-adic topology for any prime ideal $\mathfrak{p}$ in $\mathcal{O}_{K}$ dividing $\beta$. 
	\end{defn}
	For instance, when $\beta=2$ and $D_{\beta} = \{0,3\}$, the $2$-adic expansion of $1$ respect to $\{0,3\}$ is 
	$$1 = 3 + 3 \cdot 2^{1} + 3 \cdot 2^{3} + 3 \cdot 2^{5}+\cdots.$$
	It is not hard to see that the sequence $(a_i)_{i\in\N}$ is always ultimately periodic.
		
	We note that some other generalizations of $p$-adic expansion exist in the literature. K\'{a}tai \cite{Katai1999} considered number systems in rings of integers, involving sets of representatives and Peth\"{o} \cite{Petho1989} introduced number systems based on polynomials $g(t) \in \mathbb{Z}[t]$.
	
    When $\mathcal D_{\beta}$ is clear, denote
    \begin{equation}\label{2402221857}
    	(\alpha)_\beta:=(a_i)_{i\in\N}, \text{  $(\alpha)_{\beta,j}:=a_j$ ($j=0,1,\dots$).}
    \end{equation}  
    and for $b\in\mathcal D_{\beta}$ let
	\begin{equation}\label{2402221953}
		\mathcal M_b(\alpha,\beta,N):=\#\{1\le n\le N\colon (\alpha^{n})_{\beta,j}\neq b \text{ for all possible $j$}\}.
	\end{equation}

    We will see in Section \ref{sec:beta-adic} that $\beta$-adic expansion is well-defined and closely related to the radix expansion in base $\beta$, so the abuse of notations here should not cause confusion. 

    An upper bound of this $\mathcal M_b(\alpha,\beta,N)$ is also obtained.
	\begin{thm}\label{2402171658}
		Let $(\beta)=\mathfrak p_1^{e_1}\cdots\mathfrak p_{h}^{e_h}$ satisfy that $\mathfrak{p}_i$ is unramified and $N(\mathfrak{p}_i)=q_i$ for all $i$, where $q_i$ is the integer prime lying below $\mathfrak{p}_i$. If $\alpha$ is relatively prime to $\beta$, then  
		\begin{equation}\label{2404121902}
			\mathcal M_b(\alpha,\beta,N)\le  C_1 N^{\sigma(\beta)}
		\end{equation} 
		for any digit $b\in \mathcal D_{\beta}$, where $\sigma(\beta):=\frac{\log(|N(\beta)|-1)}{\log|N(\beta)|}$ and $C_1$ is a constant depending only on $\beta$.
	\end{thm}

	As a special case of Theorem \ref{2402171658}, we obtain the following generalization of Narkiewicz's result for coprime rational integers $p$ and $q$.
	
	\begin{cor}
	Let $p$, $q$ be two coprime rational integers and $b \in \{0,1,\ldots,q-1\}$. 
	Then 
	\begin{equation*}
	\mathcal M_b(p,q,N)\le  C N^{\log(q-1)/\log(q)}
	\end{equation*} 
	for some constant $C$ that can be effectively computed. 
	\end{cor} 
	
	\section{$\beta$-adic expansion} \label{sec:beta-adic}
    We begin with reviewing basic facts on algebraic number fields and $\mathfrak{p}$-adic topology. Fix a number field $K$ and let $\mathfrak{p}$ be a prime ideal of $\mathcal{O}_{K}$, we define the $\mathfrak{p}$-adic valuation and $\mathfrak{p}$-adic absolute value on the field $K$. 
	\begin{defn}\label{2404121140}
		The \emph{$\mathfrak{p}$-adic valuation} $v_\mathfrak{p}$ on $K\setminus\{0\}$ is defined as follows:
		\begin{enumerate}
			\item For each integer $a\in\mathcal{O}_K\setminus\{0\}$, let $v_\mathfrak{p}(a)$ be the unique
			non-negative integer satisfying $(a)=\mathfrak{p}^{v_\mathfrak{p}(a)}\mathfrak b$ with $\mathfrak{p}\nmid\mathfrak b$.
			\item For $x=a/b\in K\setminus\{0\}$ with $a,b\in\mathcal{O}_K$, let 
			$v_\mathfrak{p}(x):=v_\mathfrak{p}(a)-v_\mathfrak{p}(b)$.
		\end{enumerate}  
	\end{defn}
	
	\begin{rem}
		(i) It is often convenient to set $v_\mathfrak{p}(0)=+\infty$.\\
		(ii) Note that the valuation $v_\mathfrak{p}$ on $K\setminus\{0\}$ is well-defined: if $a/b=a'/b'$ for nonzero $a,b,a'$, and $b'$ in $\mathcal{O}_K$, then $v_\mathfrak{p}(a)-v_\mathfrak{p}(b)=v_\mathfrak{p}(a')-v_\mathfrak{p}(b')$.\\ 
		(iii) One can check that for all $x,y\in K$, $v_\mathfrak{p}(xy)=v_\mathfrak{p}(x)+v_\mathfrak{p}(y)$ and \[v_\mathfrak{p}(x+y)\ge\min\{v_\mathfrak{p}(x),v_\mathfrak{p}(y)\}.\]
	\end{rem}
	A prime ideal $\mathfrak{p}$ is called \emph{ramified}, if the unique integer prime $q\in\mathfrak{p}$ satisfies that $v_{\mathfrak{p}}(q)>1$. There are only finitely many ramified primes in $\mathcal O_K$. 
	
	\begin{defn} \label{c}
		The \emph{$\mathfrak{p}$-adic absolute value} $|\cdot|_\mathfrak{p}$ on the field $K$ is defined as follows: fix a constant $c\in(0,1)$, set $|\alpha|_\mathfrak{p}=c^{v_\mathfrak{p}(\alpha)}$ for $\alpha\in K\setminus\{0\}$, and $|0|_\mathfrak{p}=0$. The \emph{$\mathfrak{p}$-adic topology} on $K$ is the topology induced by $|\cdot|_\mathfrak{p}$.
	\end{defn}	
	
	 For any $\beta \in K$, let $\beta_1,\dots,\beta_s$ be the roots of the minimal polynomial of $\beta$, then the norm of $\beta$ is $N(\beta):=\big(\prod_{i=1}^{s}\beta_i\big)^{[K:\Q(\beta)]}$. For an ideal $\mathfrak a\subseteq\mathcal{O}_K$, define its norm by $N(\mathfrak a):=\#(\mathcal{O}_K/\mathfrak a)$. For principal ideals, we have $N(\beta\mathcal{O}_K)=|N(\beta)|$; see \cite[Theorem 76]{Hecke1981}. In Theorem \ref{2402171658}, we consider prime ideals $\mathfrak{p}_i$ whose norms are integer primes, this means they all have inertial degree $1$.

	\begin{defn}
		Let $G$ be an abelian group and $H\subseteq G$ be a subgroup. We say that a subset $S\subseteq G$ is a \emph{set of representatives} of the quotient group $G/H$ if the map $S\rightarrow G/H\colon x\mapsto x+H$ is a bijection.  
	\end{defn}
	
	\begin{rem}
		A set of representatives is usually not unique. In this article, we only consider the case that $G/H$ is finite, so there always exists a set of representatives. In general, if one assumes the axiom of choice, then every quotient group have sets of representatives. \\
	\end{rem}
	
	From now on, let $\mathcal D_{\beta}$ denotes a set of representatives of $\mathcal{O}_{K}/\beta\mathcal{O}_K$, then there is a natural bijection from $\mathcal D_{\beta}^i$ to $\mathcal{O}_{K}/\beta^i\mathcal{O}_K$.
	
	\begin{lem}\label{2402171219}
		For any $i\ge1$, the map
		\begin{align*}
			\mathcal D_{\beta}^i&\rightarrow\mathcal{O}_{K}/\beta^i\mathcal{O}_K \\
			(a_0,a_1,\dots,a_{i-1})&\mapsto a_0+a_1\beta+\cdots+a_{i-1}\beta^{i-1}+\beta^i\mathcal{O}_K
		\end{align*}
		is a bijection.
	\end{lem}
	\begin{proof}
		When $i=1$, the statement holds since $\mathcal D_{\beta}$ is a set of representatives of $\mathcal{O}_{K}/\beta\mathcal{O}_K$. Assume the statement is valid for $i$ and we are going to prove it for $i+1$. Note that $\#(\mathcal D_{\beta}^{i+1})=|N(\beta)|^{i+1}=|N(\beta^{i+1})|=\#\big(\mathcal{O}_{K}/\beta^{i+1}\mathcal{O}_K\big)$,
		so it suffice to prove that the map is injective.
		
		Suppose $(a_0,\dots,a_i)$ and $(b_0,\dots,b_i)$ have the same image under the map, that is, 
		$\sum_{j=0}^{i}a_j\beta^j+\beta^{i+1}\mathcal{O}_K=\sum_{j=0}^{i}b_j\beta^j+\beta^{i+1}\mathcal{O}_K$.
		Then
		\[\sum_{j=0}^{i-1}(a_j-b_j)\beta^j+(a_i-b_i)\beta^i\in\beta^{i+1}\mathcal{O}_K,\]
		so $\sum_{j=0}^{i-1}(a_j-b_j)\beta^j\in\beta^i \mathcal{O}_K$.
		By the induction hypothesis, we have $(a_0,\dots,a_{i-1})=(b_0,\dots,b_{i-1})$. Therefore, we obtain $a_i\beta^i-b_i\beta^i\in\beta^{i+1}\mathcal{O}_K$, hence $a_i-b_i\in\beta\mathcal{O}_K$. Since $a_i,b_i\in\mathcal D_{\beta}$, their difference can not lie in $\beta\mathcal{O}_K$ unless they are the same. 
	\end{proof}
	
	For each $i\ge1$, we have two natural maps $\mathcal D_{\beta}^{i+1}\rightarrow\mathcal D_{\beta}^i$ and $\mathcal{O}_{K}/\beta^{i+1}\mathcal{O}_K\rightarrow\mathcal{O}_{K}/\beta^i\mathcal{O}_K$, and it is easy to check that they commute with the map in Lemma \ref{2402171219}. Taking inverse limits, we have a bijection 
	\begin{align*}
		\mathcal D_{\beta}^{\N}&\longleftrightarrow\mathcal{O}_{K,\beta}:=\lim_{\leftarrow}\mathcal{O}_{K}/\beta^i\mathcal{O}_K  \\
		(a_i)_{i\in\N}&\longleftrightarrow a_0+a_1\beta+\cdots+a_{i}\beta^{i}+\cdots
	\end{align*}
	
	Now $\mathcal{O}_{K}$ can be viewed as a subring of $\mathcal{O}_{K,\beta}$ via natural embedding, so for every $\alpha\in\mathcal{O}_{K}$, there is a sequence $(a_i)_{i\in\N}\in\mathcal D_{\beta}^{\N}$ such that
		\[\alpha-(a_0+a_1\beta+\cdots+a_{i}\beta^{i})\in\beta^{i+1}\mathcal{O}_K\]
		holds for each $i\in\N$, and we have 	
	\begin{equation}
				\alpha=\lim_{i\rightarrow\infty}a_0+\cdots+a_i\beta^i
	\end{equation}
	with respect to $\mathfrak{p}$-adic topology for any prime ideal $\mathfrak{p}$ in $\mathcal{O}_{K}$ dividing $\beta$. If there exists another sequence $(b_i)_{i\in\N}\in\mathcal D_{\beta}^{\N}$ such that
	\[\alpha=\lim_{i\rightarrow\infty}b_0+b_1\beta+\cdots+b_i\beta^i\]
	with respect to $\mathfrak{p}$-adic topology for some prime ideal $\mathfrak{p}$ in $\mathcal{O}_{K}$ dividing $\beta$, then bijectivity implies $(a_i)_{i\in\N}=(b_i)_{i\in\N}$. Therefore we conclude that $\beta$-adic expansion is well-defined.
	
	Next we investigate the relation between $\beta$-adic expansion and the radix expansion of base $\beta$. When $\mathcal{D}_\beta=\{0,1,\dots,|N(\beta)|-1\}$ and $(\beta,\,\mathcal D_{\beta})$ is a CNS, for any $\alpha \in \mathcal{O}_{K}$, we have $\alpha = a_{0}+a_{1}\beta+\cdots+a_{m} \beta^{m}$ for some $a_{0}, \ldots, a_{m} \in \mathcal{D}$. Therefore if $\mathcal{D}_\beta$ is a set of representatives, then the $\beta$-adic expansion of $\alpha$ with respect to is $(a_{0}, \ldots, a_{m},0,0,0,\ldots)$.

	\begin{lem}\label{2402161455}
		Let $(\beta)=\mathfrak p_1^{e_1}\cdots\mathfrak p_{h}^{e_h}$. If $(\beta,\,\{0,1,\dots,|N(\beta)|-1\})$ is a CNS in $\mathcal{O}_{K}$, then $|N(\mathfrak{p}_i)|=q_i$ for all $i$, where $q_i$ is the integer prime that lies below $\mathfrak{p}_i$. 
	\end{lem}	
	\begin{proof}
		Assume $|N(\mathfrak{p}_i)|>q_i$ for some i. Then 
		\[|N(\beta)|=\prod_{1\le i\le h}\Big|N(\mathfrak{p}_i)\Big|^{e_i}>\prod_{1\le i\le h}q_i^{e_i},\]
		hence $\prod_{1\le i\le h} q_i^{e_i}\in\mathcal \{0,1,\dots,|N(\beta)|-1\}$. Since $\mathfrak{p}_i\mid q_i$ for all $i$, we have $\beta\mid\prod_{1\le i\le h} q_i^{e_i}$, thus the map 
		\begin{align*}
			\{0,1,\dots,|N(\beta)|-1\}&\rightarrow\mathcal{O}_{K}/\beta\mathcal{O}_K \\
			x&\mapsto x+\beta\mathcal{O}_K
		\end{align*}
		is not injective. Note that $\#\mathcal D_{\beta}=|N(\beta)|=\#(\mathcal{O}_{K}/\beta\mathcal{O}_K)$, so the above map is also not surjective. Therefore we can choose an element $\alpha\in\mathcal{O}_K$ such that $\alpha\not\equiv x\;(\mo1 \beta)$ for any $x\in\mathcal D_{\beta} $. However, since $(\beta,\,\{0,1,\dots,|N(\beta)|-1\})$ is a CNS in $\mathcal{O}_{K}$, we have $\alpha=c_0+c_1\beta+\cdots+c_m\beta^m$ for some $c_0,\dots,c_m\in\mathcal D_{\beta}$, which implies $\alpha\equiv c_0\;(\mo1 \beta)$, a contradiction.
	\end{proof}	
	
	\begin{lem} \label{lem:CNS}
	Let $(\beta)=\mathfrak p_1^{e_1}\cdots\mathfrak p_{h}^{e_h}$ and $\mathcal{D}_\beta=\{0,1,\dots,|N(\beta)|-1\}$. If $(\beta,\,\mathcal D_{\beta})$ is a CNS in $\mathcal O_K$, and $\mathfrak{p}_j$ is unramified for all $j$, then $\mathcal{D}_\beta$ is a set of representatives of $\mathcal{O}_{K}/\beta\mathcal{O}_K$. 
	\end{lem}
	\begin{proof}
	Note that $\#(\mathcal{O}_{K}/\beta\mathcal{O}_K)=N(\beta\mathcal{O}_K)=|N(\beta)|=\#\mathcal{D}_\beta$, so we only need to show $x-y\notin\beta\mathcal{O}_K$ for all distinct $x,y\in\mathcal{D}_\beta$. Assume that $\beta\mid x-y$ for some distinct $x,y\in\mathcal{D}_\beta$, then $\mathfrak p_1^{e_1}\cdots\mathfrak p_{h}^{e_h}\mid x-y$. Since $\mathfrak{p}_j$ is unramified for all $j$, this implies $q_1^{e_1}\cdots q_{h}^{e_h}\mid x-y$. Hence, combined with Lemma \ref{2402161455},   \[|N(\beta)|=N(\mathfrak{p}_1^{e_1}\cdots\mathfrak{p}_h^{e_h})=N(\mathfrak{p}_1)^{e_1}\cdots N(\mathfrak{p}_h)^{e_h}=q_1^{e_1}\cdots q_h^{e_h}\mid x-y,\]
	a contradiction. 		
	\end{proof}
			
	\begin{cor}[=Theorem \ref{thm:2402262104}] \label{2402262104}
		Take $\mathcal{D}_\beta=\{0,1,\dots,|N(\beta)|-1\}$. If $(\beta,\,\mathcal D_{\beta})$ is a CNS, $\beta$ is not divided by ramified primes and $\alpha$ is relatively prime to $\beta$, then \eqref{2404121902} holds for any digit $b\in\{1,\dots,|N(\beta)|-1\}$.
	\end{cor}	
	\begin{proof}
	Note that if the radix expansion of $\alpha^n$ in base $\beta$ is 		$$\alpha^n=a_0+a_1\beta+\cdots+a_m \beta^m,\quad a_j\in D_{\beta} \;(j=0,1,\dots,m),$$
	then we may add infinitely many zeroes to obtain its $\beta$-adic expansion $$(a_{0}, a_{1}, \ldots, a_{m}, 0,0, \ldots).$$ Therefore this corollary follows from Theorem \ref{2402171658}, Lemma \ref{2402161455} and Lemma \ref{lem:CNS}.
	\end{proof}
	
	\begin{rem}
When $b=0$, if the length of the radix expansion of $\alpha^n$ is long enough, the we may use a similar inequality as \eqref{2402261744} and follow the proof of Theorem \ref{2402171658} to deduce the desired bound. 
	\end{rem}
		
	\section{$\mathfrak{p}$-adic interpolation of the sequence $(\alpha^n)_{n\in\N}$}
	Let $(\beta)=\mathfrak p_1^{e_1}\cdots\mathfrak p_{h}^{e_h}$ be the prime ideal decomposition of $(\beta)$ and $\alpha\in\mathcal O_K$ be relatively prime to $\beta$. Fix a prime ideal $\mathfrak{p}\in\{\mathfrak{p}_1,\dots,\mathfrak{p}_h\}$. In order to analyze the $\beta$-adic expansion of $\alpha^n$, we need to introduce a powerful method called $\mathfrak{p}$-adic interpolation.

	 Recall that $(K,|\cdot|_{\mathfrak{p}})$ is a valued field and the distance of $x,y\in K$ is defined as $|x-y|_{\mathfrak{p}}$. A valued field is said to be \emph{complete} when every Cauchy sequence has a limit.
	
	\begin{prop}[{\cite[Chapter 1, (M)]{Ribenboim1999}}]
		Every valued field has a completion.
	\end{prop}
	
	We denote the completion of $K$ with respect to the $\mathfrak{p}$-adic absolute value $|\cdot|_\mathfrak{p}$ by $K_\mathfrak{p}$, and denote the extended absolute value again by $|\cdot|_\mathfrak{p}$. Let 
	\[\bar{B}(0,1)=\{x\in K_{\mathfrak{p}}\colon |x|_\mathfrak{p}\le1\}\]
	denote the closed unit ball of $K_{\mathfrak{p}}$. It is clear that $\mathcal{O}_K\subset \bar{B}(0,1)$. 
	
	Let $(\alpha_n)_{n\in\N}$ be a sequence of integers in $\mathcal{O}_K$. A \emph{$\mathfrak{p}$-adic interpolation} of the sequence $(\alpha_n)_{n\in\N}$ is a continuous function $G(x)$, defined in the unit ball $\bar{B}(0,1)$, with $G(n)=\alpha_n$ for all $n\in\N$. 
	
	\begin{lem}\label{2402211316}
		If $\alpha\in\mathcal{O}_K$ satisfies that $\mathfrak{p}\nmid\alpha$, then there is a rational integer $u_{\mathfrak{p}}$ such that the sequence $(\alpha^n)_{n\in\N}$ can be divided into subsequences \[(\alpha^{l}(\alpha^{u_{\mathfrak{p}}})^n)_{n\in\N},\quad l=0,1,\dots,u_{\mathfrak{p}}-1,\] 
		and for each $l$, the sequence $(\alpha^{l}(\alpha^{u_{\mathfrak{p}}})^n)_{n\in\N}$ has an analytic $\mathfrak{p}$-adic interpolation $G_l$. 
	\end{lem}
	
	\begin{proof}
		Define the formal series
		\[\log(1+X):=\sum_{n=1}^{\infty}(-1)^{n+1}\frac{X^n}{n}\]
		Recall that for a power series $f(X)=\sum_{n=0}^{\infty}a_nX^n$ with coefficients in $K_{\mathfrak{p}}$, the radius of convergence is defined as
		\begin{equation}\label{2402191447}
			r=1/(\limsup_{n\rightarrow+\infty}|a_n|_{\mathfrak{p}}),
		\end{equation}
		then $f(x)$ converges for every $x\in K_{\mathfrak{p}}$ with $|x|_{\mathfrak{p}}<r$, see \cite[Proposition 5.4.1]{Gouvea[2020]copyright2020} for details. 
		
		Let $a_n=(-1)^{n+1}/n$, we claim that  $|a_n|_{\mathfrak{p}}^{1/n}=c^{-v_{\mathfrak{p}}(n)/n}\rightarrow1$ as $n\rightarrow\infty$, where $c$ is the constant fixed in Definition \ref{c}. To see this, let $q$ be the unique integer prime lying below $\mathfrak{p}$ and $n=q^{v_q(n)}a$ with $q\nmid a$, then $v_q(n)\le\log n$; on the other hand,
		\begin{equation}\label{2402191504} (n)=(q)^{v_q(n)}(a)=(\mathfrak{p}^{v_{\mathfrak{p}}(q)}\mathfrak{b})^{v_q(n)}(a),
		\end{equation}
		with $\mathfrak{p}\nmid\mathfrak{b}$, thus
		\begin{equation}\label{2402191505}
			v_{\mathfrak{p}}(n)=v_{\mathfrak{p}}(q)v_q(n)\le v_{\mathfrak{p}}(q)\log n,
		\end{equation}
		which completes the proof of the claim. Hence, applying \eqref{2402191447}, we can define the $\mathfrak{p}$-adic logarithm of $x\in B(1,1):=\{x\in K_\mathfrak{p}\colon |x-1|_{\mathfrak{p}}<1\}$ as
		\[\log_\mathfrak{p}(x)=\log_\mathfrak{p}(1+(x-1))=\sum_{n=1}^{\infty}(-1)^{n+1}\frac{(x-1)^n}{n}.\]
		
		Define the formal series $\bexp(X):=\sum_{n=0}^{\infty}X^n/n!$. To calculate the radius of convergence \eqref{2402191447}, let $a_n=1/(n!)$. Similar to \eqref{2402191504} and \eqref{2402191505}, we have
		\begin{equation}\label{2402211409}
			v_{\mathfrak{p}}(n!)=v_{\mathfrak{p}}(q)v_q(n!)< \frac{v_{\mathfrak{p}}(q)n}{q-1},
		\end{equation}
		where we use $v_q(n!)\le n/(q-1)$ in the last inequality (see \cite[Lemma 5.7.4]{Gouvea[2020]copyright2020}). Hence, \[|1/n!|_{\mathfrak{p}}^{1/n}=c^{-v_{\mathfrak{p}}(n!)/n}<c^{-v_{\mathfrak{p}}(q)/(q-1)}, \]
		this implies that the radius of convergence $r\ge c^{v_{\mathfrak{p}}(q)/(q-1)}$. Therefore, we can define the $\mathfrak{p}$-adic exponential function as
		\begin{equation}\label{2402201226}
			\bexp_{\mathfrak{p}}(x):=\sum_{n=0}^{\infty}\frac{x^n}{n!},\quad x\in B(0,c^{v_{\mathfrak{p}}(q)/(q-1)}) 
		\end{equation}
		
		Observe that for all $x\in B(1,c^{v_{\mathfrak{p}}(q)})$ and $N\in\N$, 
		\begin{align}
			\Big|\sum_{n=1}^{N}(-1)^{n+1}\frac{(x-1)^n}{n}\Big|_{\mathfrak{p}}&\le\max_{1\le n\le N}\frac{|x-1|_{\mathfrak{p}}^n}{|n|_{\mathfrak{p}}}\\
			&=\max_{1\le n\le N}\Big(c^{nv_{\mathfrak{p}}(x-1)}/c^{v_{\mathfrak{p}}(n)}\Big) \notag\\
			&\le\max_{1\le n\le N}\Big(c^{nv_{\mathfrak{p}}(x-1)}/c^{v_{\mathfrak{p}}(q)\log n}\Big)\le|x-1|_{\mathfrak{p}},  \notag
		\end{align}
		where we use \eqref{2402191505} in the third step. Hence, for each $x\in B(1,c^{v_{\mathfrak{p}}(q)})$, we have
		\begin{equation}\label{2402201227}
			|\log_{\mathfrak{p}}(x)|_{\mathfrak{p}}\le|x-1|_{\mathfrak{p}}.
		\end{equation}
		Thus, by \eqref{2402201226} and \eqref{2402201227},
		$\log_{\mathfrak{p}}(x)$ is in the domain of $\bexp_{\mathfrak{p}}$ when $x\in \bar{B}(1,c^{v_{\mathfrak{p}}(q)+1})$.
		
		Note that $\mathcal{O}_K/\mathfrak{p}^{v_{\mathfrak{p}}(q)+1}$ is a finite additive group, the sequence $\alpha,\alpha^2,\dots,\alpha^n,\dots$ 
		must satisfy that there exist two integers $n,m$ with $0\le n<m$ such that  \[\alpha^m+\mathcal{O}_K/\mathfrak{p}^{v_{\mathfrak{p}}(q)+1}=\alpha^n+\mathcal{O}_K/\mathfrak{p}^{v_{\mathfrak{p}}(q)+1},\]
		thus $\alpha^{n}(\alpha^{m-n}-1)\in\mathfrak{p}^{v_{\mathfrak{p}}(q)+1}$. 
		By the condition $\mathfrak{p}\nmid \alpha$, we have $\alpha^{m-n}-1\in\mathfrak{p}^{v_{\mathfrak{p}}(q)+1}$. Therefore, there is an integer $u_{\mathfrak{p}}$ such that $|\alpha^{u_{\mathfrak{p}}}-1|_{\mathfrak{p}}\le c^{v_{\mathfrak{p}}(q)+1}$. 
		
		By \eqref{2402201227}, for $|x|_{\mathfrak{p}}\le1$,
		\begin{equation}\label{2402211402}
			|x\log_{\mathfrak{p}}(\alpha^{u_{\mathfrak{p}}})|_{\mathfrak{p}}=|x|_{\mathfrak{p}}|\log_{\mathfrak{p}}(\alpha^{u_{\mathfrak{p}}})|_{\mathfrak{p}}\le 1\cdot |\alpha^{u_{\mathfrak{p}}}-1|_{\mathfrak{p}}\le c^{v_{\mathfrak{p}}(q)+1}.  		
		\end{equation}
		Combined with \eqref{2402201226}, we obtain that $\bexp_{\mathfrak{p}}(x\log_{\mathfrak{p}}(\alpha^{u_{\mathfrak{p}}}))$ is well-defined on the closed ball $\bar{B}(0,1)$. This expression will serve as the definition of $(\alpha^{u_{\mathfrak{p}}})^{x}$, $x\in\bar{B}(0,1)$. 
		
		Now take $G_l(x)=\alpha^l(\alpha^{u_{\mathfrak{p}}})^{x}$, $x\in\bar{B}(0,1)$, for $l\in\{0,1,\dots,u_{\mathfrak{p}}-1\}$, which is the analytic $\mathfrak{p}$-adic interpolation that we want. 
	\end{proof}
	
	\begin{cor}\label{2402221148}
		Let $(\beta)=\mathfrak p_1^{e_1}\cdots\mathfrak p_{h}^{e_h}$ and $\alpha\in\mathcal O_K$ be relatively prime to $\beta$. Let $u_{\mathfrak p_i}$ ($i=1,2,\dots,h$) be as in Lemma \ref{2402211316} and $u=\prod_{i=1}^{h}u_{\mathfrak p_i}$. Then $G_l(x):=\alpha^l(\alpha^{u})^{x}$ is an analytic $\mathfrak{p}_i$-adic interpolation of $(\alpha^{l}(\alpha^{u})^n)_{n\in\N}$ for all $i=1,2,\dots,h$ and $l=0,1,\dots,u-1$. 
	\end{cor}
	\begin{proof}
		Let $q_i$ be the unique integer prime lying below $\mathfrak{p}_i$ for $i=1,2,\dots,h$. By the definition of $u_{\mathfrak{p}_i}$, we have
		\[\alpha^{u_{\mathfrak{p}_i}}+\mathcal{O}_K/\mathfrak{p}_i^{v_{\mathfrak{p}_i}(q_i)+1}=1+\mathcal{O}_K/\mathfrak{p}_i^{v_{\mathfrak{p}_i}(q_i)+1}.\]
		Hence, for all $i\in\{1,2,\dots,h\}$,
		$\alpha^{u}+\mathcal{O}_K/\mathfrak{p}_i^{v_{\mathfrak{p}_i}(q_i)+1}=1+\mathcal{O}_K/\mathfrak{p}_i^{v_{\mathfrak{p}_i}(q_i)+1}$,
		that is, 
		\begin{equation}\label{2402221234}
			|\alpha^{u}-1|_{\mathfrak{p}_i}\le c^{v_{\mathfrak{p}_i}(q_i)+1}.
		\end{equation}
		Therefore, $\bexp_{\mathfrak{p}_i}(x\log_{\mathfrak{p}_i}(\alpha^{u}))$ is well-defined for $|x|_{\mathfrak{p}_i}\le1$ (the discussion is similar to the one in the proof of Lemma \ref{2402211316}), and this expression will serve as the definition of $(\alpha^{u})^{x}$,  $|x|_{\mathfrak{p}_i}\le1$.
	\end{proof}
	
	\begin{rem}
		One can think of $G_l(x)=\alpha^l(\alpha^{u})^x$ as a formal function, which is well-defined on the closed ball $\bar{B}(0,1)$ in $K_{\mathfrak{p}}$ for all $\mathfrak{p}\in\{\mathfrak{p}_1,\dots,\mathfrak{p}_h\}$. 
	\end{rem}
	
	\section{Proof of Theorem \ref{2402171658}}
	We begin with a simple lemma. 
	
	\begin{lem}\label{2402221542}
		Let $G_l(x)=\alpha^l(\alpha^{u})^x$ and $\{\mathfrak{p}_1,\dots,\mathfrak{p}_h\}$ be as in Corollary \ref{2402221148}. Then there exist integers $n_0,m_0$ such that
		\begin{equation}\label{2402261329}
			|G_l(x)-G_l(y)|_{\mathfrak{p}}\ge c^{n_0}|x-y|_{\mathfrak{p}},
		\end{equation}
		for all $x,y$ with $|x-y|_{\mathfrak{p}}\le c^{m_0}$ and $\mathfrak{p}\in\{\mathfrak{p}_1,\dots,\mathfrak{p}_h\}$.
	\end{lem}
	\begin{proof}
Fix a $\mathfrak{p}\in\{\mathfrak{p}_1,\dots,\mathfrak{p}_h\}$, we claim that there exist integers $n_{\mathfrak{p}},m_{\mathfrak{p}}>0$ such that for every pair of distinct $x,y\in\bm\bar B(0,1)$,  
\begin{equation}\label{2312211103}
	\text{if $|x-y|_{\mathfrak{p}}\le c^{m_{\mathfrak{p}}}$, then }|G_l(x)-G_l(y)|_{\mathfrak{p}}\ge c^{n_{\mathfrak{p}}}|x-y|_{\mathfrak{p}}.
\end{equation}
Assume that for every $n$, there is a pair of distinct points $x_n$, $y_n$ satisfying
\begin{equation}\label{2312211045}
	|x_n-y_n|_{\mathfrak{p}}\le\frac{1}{n},\quad |G_l(x_n)-G_l(y_n)|_{\mathfrak{p}}<\frac{1}{n}|x_n-y_n|_{\mathfrak{p}}.
\end{equation}
Since $\bm\bar B(0,1)$ is compact (similar to \cite[Corollary 4.2.7]{Gouvea[2020]copyright2020}), $(x_n)_{n\ge1}$ has a convergent subsequence $(x_{n_j})_{j\ge1}$, we assume that $x_{n_j}\rightarrow x_0$. We must have $y_{n_j}\rightarrow x_0$ as well. Suppose that $G_l(z)=\sum_{n=0}^{\infty}c_n z^n$ since $G_l$ is analytic, then 
\begin{align}\label{2312191707}
	&\frac{G_l(x_{n_j})-G_l(y_{n_j})}{x_{n_j}-y_{n_j}}=\frac{ \sum_{n=0}^{\infty}c_n(x_{n_j}^n-y_{n_j}^n)}{x_{n_j}-y_{n_j}}\\
	=&\sum_{n=0}^{\infty}c_n(x_{n_j}^{n-1}+x_{n_j}^{n-2}y_{n_j}+\cdots+y_{n_j}^{n-1})\rightarrow\sum_{n=0}^{\infty}c_nnx_0^{n-1}=G_l'(x_0),\notag
\end{align}
as $j\rightarrow+\infty$. But, by \eqref{2312211045}, 
\[\Big|\frac{G_l(x_n)-G_l(y_n)}{x_n-y_n}\Big|_p<\frac{1}{n},\]
combined with \eqref{2312191707}, we have $G_l'(x_0)=0$. However, this is impossible, one can check that:  
\begin{align}
	G_l'(x)&=\alpha^{l}\Big(\sum_{n=0}^{\infty}\frac{\big(x\log_{\mathfrak{p}}(\alpha^{ u})\big)^{n}}{n!}\Big)' \notag\\
	&=\alpha^{l}\sum_{n=1}^{\infty}\frac{\big(\log_{\mathfrak{p}}(\alpha^{ u})\big)^{n}x^{n-1}}{(n-1)!} \notag\\
	&=\big(\alpha^{l}\log_{\mathfrak{p}}(\alpha^{u})\big)(\alpha^{u})^x\neq0,  
\end{align}
for all $x\in\bm\bar B(0,1)$. This completes the proof of the claim.

Let $m_{0}=\max\big\{m_{\mathfrak{p}}\colon\mathfrak{p}\in\{\mathfrak{p}_1,\dots,\mathfrak{p}_h\}\big\}$ and $n_{0}=\max\big\{n_{\mathfrak{p}}\colon\mathfrak{p}\in\{\mathfrak{p}_1,\dots,\mathfrak{p}_h\}\big\}$. This completes the proof of the lemma.
	\end{proof}

	Fix a digit $b\in\mathcal D_{\beta}$, for a word $(a_j)_{j=0}^{k-1}\in(\mathcal D_{\beta}\setminus\{b\})^k$, denote
	\[
	\big[(a_j)_{j=0}^{k-1}\big]^{(l)}
	:=\big\{0\le n\le |N(\beta)|^k-1\colon \big(\alpha^l(\alpha^u)^n\big)_{\beta,j}=a_j\text{ for } j=0,\dots,k-1\big\},
	\]
	where $u$ is as in Corollary \ref{2402221148}, $l=0,1,\dots,u-1$ and the definition of $(\cdot)_{\beta,j}$ is as in \eqref{2402221857}. By the definition of $\mathcal M_b\big(\alpha,\beta,u|N(\beta)|^k\big)$ (see \eqref{2402221953}), we have 
	\begin{align}\label{2402261744}
		\mathcal M_b\big(\alpha,\beta,u|N(\beta)|^k\big)
		\le&\#\big\{1\le n\le u|N(\beta)|^k\colon (\alpha^{n})_{\beta,j}\neq b \text{ for } j=0,\dots,k-1\big\}\notag\\
		\le&\sum_{l=1}^{u}\sum_{(a_i)_{i=0}^{k-1}\in\big(\mathcal D_{\beta}\setminus\{b\}\big)^k}\#\big[(a_j)_{j=0}^{k-1}\big]^{(l)}. 
	\end{align}
	
	We are now going to estimate $\# \big[(a_j)_{j=0}^{k-1}\big]^{(l)}$. Let $q_j$ be the unique integer prime lying below $\mathfrak{p}_j$ for $j=1,2,\dots,h$. By the condition $N(\mathfrak{p}_j)=q_j$, we have
	\begin{equation}\label{2404191406}
		|N(\beta)|=\prod_{j=1}^{h}N(\mathfrak{p}_j)^{e_j}=\prod_{j=1}^{h}q_j^{e_j}. 
	\end{equation}
	Consider the partition
	\begin{equation}\label{2404191442}
		\big[(a_j)_{j=0}^{k-1}\big]^{(l)}=\bigcup_{i=0}^{|N(\beta)|^{m_0}-1}\big[(a_j)_{j=0}^{k-1}\big]^{(l)}_{i}
	\end{equation}
	where 
	\[\begin{split}
		\big[(a_j)_{j=0}^{k-1}\big]^{(l)}_{i}:=\Big\{0\le n\le |N(\beta)|^k-1\colon& n\in\big[(a_j)_{j=0}^{k-1}\big]^{(l)}\\
		& \text{ and } n\equiv i\,(\mo1 |N(\beta)|^{m_0})\Big\}, 
	\end{split}\]
	and $m_0$ is defined as in Lemma \ref{2402221542}. 
	Suppose that $n,m$ are in $\big[(a_j)_{j=0}^{k-1}\big]^{(l)}_i$, then 
	\[\alpha^l(\alpha^u)^n+\beta^k\mathcal{O}_K=\sum_{i=0}^{k-1}a_i\beta^i+\beta^k\mathcal{O}_K=\alpha^l(\alpha^u)^m+\beta^k\mathcal{O}_K,\]
	that is, $\beta^k\mid \big(\alpha^l(\alpha^u)^n-\alpha^l(\alpha^u)^m\big)$. Recall that $(\beta)=\mathfrak p_1^{e_1}\cdots\mathfrak p_{h}^{e_h}$, we have 
	\begin{equation}\label{2402261328}
		\Big|\alpha^l(\alpha^u)^n-\alpha^l(\alpha^u)^m\Big|_{\mathfrak{p}_j}\le c^{ke_j},
	\end{equation}
	for all $j\in\{0,1,\dots,h\}$. Moreover, $n,m\in\big[(a_j)_{j=0}^{k-1}\big]^{(l)}_i$ also implies that 
	\[n\equiv m\,(\mo1 |N(\beta)|^{m_0}),\] 
	that is, $|N(\beta)|^{m_0}\mid n-m$. Combined this with \eqref{2404191406}, we have
	\[|n-m|_{\mathfrak{p}_j}\le c^{m_0},\]	
	for all $j\in\{0,1,\dots,h\}$. Hence, by Lemma \ref{2402221542} and \eqref{2402261329}, we have 
	\[|n-m|_{\mathfrak{p}_j}\le c^{-n_0}c^{ke_j},\]
	for all $j\in\{0,1,\dots,h\}$. This implies that  $\prod_{j=1}^{h}\mathfrak{p}_j^{ke_j-n_0}|(n-m)$. Recall the condition that $\mathfrak{p}_j$ is unramified for all $j$, this implies that 
	\[\prod_{j=1}^{h}q_j^{ke_j-n_0}\mid (n-m),\]
	where $q_i$ is the integer prime that lies in $\mathfrak{p}_i$. Hence, the distance \[|n-m|\ge\prod_{j=1}^{h}q_j^{ke_j-n_0},\] 
	which holds for each pair of distinct $n,m\in\big[(a_j)_{j=0}^{k-1}\big]_i^{(l)}$. Therefore,
	\begin{equation}\label{2402261724}
		\#\big[(a_j)_{j=0}^{k-1}\big]_i^{(l)}\le|N(\beta)|^k/\Big(\prod_{j=1}^{h}q_j^{ke_j-n_0}\Big).
	\end{equation}
	On the other hand, by \eqref{2404191406}, 
	\begin{equation}\label{2402261723}
		|N(\beta)|^k=\prod_{j=1}^{h}q_j^{ke_j}.
	\end{equation}
     Applying \eqref{2402261723} to \eqref{2402261724}, we have
	\begin{equation}\label{2402261746}
		\#\big[(a_j)_{j=0}^{k-1}\big]_i^{(l)}\le \widetilde{C}_0,
	\end{equation}
	where $\widetilde{C}_0:=\big(\prod_{j=1}^{h}q_j\big)^{n_0}$. Applying \eqref{2402261746}, \eqref{2404191442} to \eqref{2402261744}, we obtain
	\begin{equation}\label{2402261807}
		\mathcal M_b\big(\alpha,\beta,u|N(\beta)|^k\big)\le C_0\cdot\big|N(\beta)\big|^{k\sigma(\beta)}
	\end{equation}
	where $\sigma(\beta):=\frac{\log(|N(\beta)|-1)}{\log|N(\beta)|}$ and $C_0:=u|N(\beta)|^{m_0}\widetilde{C}_0$. For an integer $N\in\N$, there is an integer $k\in\N$ such that $|N(\beta)|^{k-1}\le N\le |N(\beta)|^k$; then, by \eqref{2402261807}, we have
	\[\begin{split}
		\mathcal M_b\big(\alpha,\beta,N\big)\le\mathcal M_b\big(\alpha,\beta,u|N(\beta)|^k\big)\le C_1 N^{\sigma(\beta)},
	\end{split}\]
	where $C_1:=C_0|N(\beta)|^{\sigma(\beta)}$. This completes the proof.
	
	\section*{Acknowledgements}
	We thank Wladyslaw Narkiewicz for pointing out errors in the previous version of this article. We thank the referee for helpful suggestions. This work was supported in part by NSFC No. 12471085, Science and Technology Commission of Shanghai Municipality (STCSM) No. 22DZ2229014, NSFC No. 12401006, and Guangdong Basic and Applied Basic Research Foundation No.  2023A1515110272.
	
	\section*{Declaration of interests}
	There are no relevant financial or non-financial competing interests to report.

\end{document}